\documentclass[11pt]{amsart}
\usepackage{array, amsmath, enumerate, color}
\usepackage{amssymb}
\usepackage{mathrsfs}
\usepackage{graphicx,subfigure}
\newcommand*{\rom}[1]{\expandafter\@slowromancap\romannumeral #1@}

\makeatletter
 \usepackage{amsthm}
\usepackage{fullpage}

\makeatother

\newtheorem{thm}{Theorem}[section]
\newtheorem{Lemma}[thm]{Lemma}

\newtheorem{conj}[thm]{Conjecture}

\newtheorem{definition}[thm]{Definition}

\title{The strong chromatic index of $(3,\Delta)$-bipartite graphs}
\author[Huang]{Mingfang Huang}
\address{Department of Mathematics, School of Science,
Wuhan University of Technology, China}
\thanks{The first author is supported by the Fundamental Research Funds for the Central Universities (WUT: 2015IA002).}

\author[YU]{Gexin Yu}
\address{Department of Mathematics,
The College of William and Mary,
Williamsburg, VA 23185}
\thanks{The second author is supported in part by the NSA grant: H98230-16-1-0316.}  

\author[Zhou]{Xiangqian Zhou}
\address{Department of Mathematics and Statistics, Wright State University, Dayton, Ohio, 45435}
\email{xiangqian.zhou@wright.edu}

\date{\today}

\begin{document}

\begin{abstract} A strong edge-coloring of a graph $G=(V,E)$ is a partition of its edge set $E$ into induced matchings. 
We study bipartite graphs with one part having maximum degree at most $3$ and the other part having maximum degree $\Delta$. We show that every such graph has a strong edge-coloring using at most $3 \Delta$ colors. Our result confirms a conjecture of Brualdi and Quinn Massey ~\cite{[BQ]} for this class of bipartite graphs.

\end{abstract}
\maketitle

\section{ introduction}

Graphs in this article are assumed to be simple and undirected. Let $G$ be a simple undirected graph. A {\em proper edge-coloring} of $G$ is an assignment of colors to the edges such that no two adjacent  edges have the same color. Clearly, every coloring class is a matching of $G$. However, these matchings may not be induced. If one requires each color class to be an induced matching, that leads to the notion of strong edge-coloring, first introduced by Fouquet and Jolivet ~\cite{[FJ]}.  A {\em strong edge-coloring} of a graph $G$ is a proper edge-coloring such that every two edges joined by another edge are colored differently.
In a strong edge-coloring, every color class induces a matching.
The minimum number of colors required in a strong edge-coloring of $G$ is called the {\em strong chromatic index} and is denoted by $\chi_s^\prime(G)$.

Let $e$ and $e^\prime$ be two edges of $G$. We say that $e$ {\em sees} $e^\prime$ if $e$  and $e^\prime$ are adjacent or share a common adjacent edge. So, equivalently, a strong edge-coloring is an assignment of colors to all edges such that every two edges that can see each other receive distinct colors.

Let $\Delta$ be the maximum degree of $G$ and for $u  \in V(G)$, let $d_G(u)$ denote the degree of $u$ in the graph $G$. For each $S \subseteq V(G)$, let $\Delta(S)= \max\{d_G(s) : s \in S\}$. Using greedy coloring arguments, one may easily show that  $\chi_s^\prime(G) \leq 2 \Delta ^2 - 2 \Delta +1$ holds for every graph $G$. Erd\H{o}s and Ne\v{s}et\v{r}il \cite{[EN]} conjectured the following tighter upper bounds and they also gave examples of graphs that achieve these bounds.

\begin{conj} {\rm (Erd\H{o}s and Ne\v{s}et\v{r}il \cite{[EN]}) } \label{conj:any-g} For every graph $G$, the following inequalities hold. 

\begin{equation*}
 \chi_s^\prime(G) \leq
\begin{cases}
 \frac{5}{4} \Delta^2 & \text{if } \Delta \text{ is even},\\
\frac{1}{4} (5 \Delta^2 - 2 \Delta +1) & \text{if } \Delta \text{ is odd}.
\end{cases}
\end{equation*}

\end{conj}

In this paper, we study strong edge-coloring of bipartite graphs. Faudree Gy\'{a}rf\'{a}s, Schelp,  and Tuza~\cite{[FGST]} conjectured the following.

\begin{conj}
\label{conj:delta-square} {\rm (Faudree {\em et al.} \cite{[FGST]})}
For every bipartite graph $G$,   the strong chromatic index of $G$ is at most $\Delta^2$.
 \end{conj}

Steger and Yu \cite{[SY]} confirmed Conjecture~\ref{conj:delta-square}  when the maximum degree is at most $3$. Let $d_A$ and $d_B$ be two positive intergers. A {\em $(d_A, d_B)$-bipartite graph} is a bipartite graph with bipartition $A$ and $B$ such that $\Delta(A) \leq d_A$ and $\Delta(B) \leq d_B$. Brualdi and Quinn Massey \cite{[BQ]} strengthened Conjecture ~\ref{conj:delta-square} to the following.

\begin{conj}
\label{conj:dA-dB} {\rm (Brualdi and Quinn Massey \cite{[BQ]})}
If $G$ is a $(d_A, d_B)$-bipartite graph, then $\chi_s^\prime(G)\leq d_A d_B$. 
\end{conj}

Note that, the bounds given in Conjectures ~\ref{conj:delta-square} and ~\ref{conj:dA-dB}, if proven, would be tight; as the complete bipartite graph $K_{m,n}$ has strong chromatic index $mn$. 

Nakprasit~\cite{[KN]}  confirmed Conjecture ~\ref{conj:dA-dB} for the class of $(2, \Delta)$-bipartite graphs.  Recently, Bensmail, Lagoutte, and Valicov \cite{[BLV]} proved the following result.

\begin{thm}
\label{thm:4-delta} {\rm (Bensmail {\em et al.}  \cite{[BLV]})} If $G$ is  a $(3, \Delta)$-bipartite graph, then $\chi_s^\prime(G)\leq  4 \Delta$.
\end{thm}

Note that Theorem~\ref{thm:4-delta} gives a weaker bound than what is given in Conjecture ~\ref{conj:dA-dB}.  In the last section of their paper, the authors of ~\cite{[BLV]} pointed out several possible strategies to improve the bound down to $3 \Delta$. Following their suggestions, we prove the following result.

\begin{thm}
\label{thm:main}

If $G$ is a $(3, \Delta)$-bipartite graph, then $\chi_s^\prime(G)\leq3\Delta$.
\end{thm}

Our proof scheme is very similar to a scheme used in ~\cite{[BLV], [BQ], [SY]}, first introduced in ~\cite{[SY]}.  The scheme consists of using a matrix to describe a special decomposition of the graph.  One minor difference in our approach is that we do not use a matrix, but instead work directly with the decomposition. The main difference in our approach lies in two aspects:  the way we choose the decomposition of $G$ and the order in which the edges are colored. Details on each will be presented in Sections ~\ref{sec:decom} and ~\ref{sec:proof}, respectively. 

The paper is organized as follows. In Section ~\ref{sec:decom}, we define a decomposition of $G$ where $G$ is a $(3, \Delta)$-bipartite graph and we also prove some basic properties of the decomposition.  The main proof is presented in Section~\ref{sec:proof}. Finally in Section~\ref{sec:ext} we talk about some possible extensions of our result.

\section{a decomposition of $G$} \label{sec:decom}

 %For any subset $S$ of $E(G)$, $G-S$ denotes the graph obtained from $G$ by deleting all the edges of $S$ together with all the trivial subgraphs of $G$ generated in this process.

Suppose that $G$ is a $(3,\Delta)$-bipartite graph with bipartition $(A,B)$ with $\Delta(A)\leq3$. Our goal is to show that  $G$ has a strong edge-coloring using at most $3\Delta$ colors.  Nakprasit's theorem \cite{[KN]} implies that the result holds if $\Delta(A) \leq 2$ or $\Delta(B) \leq 2$. So we may assume that $\Delta(A)=3$ and that $\Delta \geq 3$. We may further assume that all vertices of $A$ are of degree exactly $3$ (for otherwise, we may add a number of degree-1 vertices to $B$ and increase the degree of every vertex of $A$ to $3$). 

Now we decompose the graph $G$ into $\Delta$ edge-disjoint spanning subgraphs $G_1, G_2, \cdots, G_{\Delta}$ such that $E= \cup_{i=1} ^{\Delta} E(G_i)$,  and  $d_{G_i}(b) \leq 1$ for each $b \in B$ and for each $i \in \{1, 2, \cdots, \Delta\}$. We call such a decomposition a {\em $B$-singular decomposition} of $G$. 

Let $G_1, G_2, \cdots, G_{\Delta}$ be a $B$-singular decomposition of $G$. For every vertex $a \in A$ and for every $i \in \{1, 2, \cdots, \Delta\}$, we have that $0 \leq d_{G_i} (a) \leq 3$.  Here we will use the notions of type-1, type-2, and type-3 vertices introduced in \cite{[BLV]} and we also require some new notions on the edges of $G$.   

\begin{definition} Let $a$ be a vertex of $A$. 
\label{d1} 
\begin{itemize}
\item If there exists $1 \leq i \leq \Delta$ with $d_{G_i}(a) =3$, then $a$ is called a {\em type-1 vertex}, and the edges incident to $a$ are called {\em triplex-edges}.
\item If there exists $1 \leq i \leq \Delta$ with $d_{G_i}(a)=2$, then $a$ is called a {\em type-2 vertex}, and the two edges of $G_i$ incident to $a$ are called {\em paired-edges}, the edge incident to $a$ that is not in $G_i$ is called a {\em lonely-edge}.

\item If there exist distinct $i$, $j$, and $k$ such that $d_{G_{i}} (a)= d_{G_{j}}(a) = d_{G_{k}} (a) = 1$, then $a$ is called a {\em type-3 vertex}, and the edges incident to $a$ are called {\em dispersed-edges}. 
\end{itemize}
\end{definition}

%\begin{enumerate}[(i)]
%\item For each $i \in \{ 1, 2,  \cdots, \Delta \}$ and for each $b \in B$, $d_{G_i}(b)\le 1$.
%\item Subject to (i), maximize the number of vertices in $A_1=\{a\in A: d_{G_i}(a)=3$, for some $1\leq i\leq\Delta\}$.
%\item Subject to (i) and (ii), maximize the number of vertices in $A_2=\{a\in A: d_{G_i}(a)=2$, for some $1\leq i\leq\Delta\}$.
%\end{enumerate}

%\begin{definition}
%\label{d1}
%Let $e$ be an edge of $G_i$. Suppose that $e$ is incident to $a\in A$.
%\begin{itemize}
%\item If $d_{G_i}(a)=1$ and $d_{G_j}(a)=2$ for some $1\leq j\neq i\leq\Delta$ then $e$ is  called a lonely-edge, and the other two edges incident to $a$ are called paired-edges.
%\item If $d_{G_i}(a)=1$ and $d_{G_j}(a)=1$ for some $1\leq j\neq i\leq \Delta$, then $e$ is  called a dispersed-edge.
%\item If $d_{G_i}(a)=3$, then $e$ is  called a triplex-edge.
%\end{itemize}
%\end{definition}

For each $1\le i\le \Delta$,  let $H_i$ be the induced subgraph of $G$ spanned by the endpoints of all lonely-edges of $G_i$. Note that $H_i$ may contain edges that are not in $G_i$.  Since $G$ is bipartite, a cycle $C$ of $H_i$ must be of even length. Suppose that $|C|=2k$. Then $k$ may be even or odd. Let  $\mathscr{C}=\{C\in C_{2k}:$ $k$ is odd and $C$ is a cycle in $H_i$ for some $1\leq i\leq \Delta\}$. We now choose  a special  $B$-singular decomposition $\mathscr{F}=\{G_1, G_2, \ldots, G_{\Delta}\}$ of $G$ as follows.
 
 \begin{enumerate}
 \item First we maximize the number of type-1 vertices; 
 \item Subject to (1), we maximize the number of type-2 vertices; 
 \item Subject to (1) and (2), we minimize the number of cycles in $\mathscr{C}$.
\end{enumerate}

Condition (3) was not required in Bensmail {\em et al.} ~\cite{[BLV]}. However, we need this condition to deal with one special case.  Note that the decomposition $\mathscr{F}$ may not be unique. 

The next three lemmas were proved implicitly in Bensmail {\em et al.} ~\cite{[BLV]} based on the matrix they used. We state these results in terms of graphs and present a separate proof for each of them.  

\begin{Lemma}
\label{le00} Let $ab$ $(a\in A, b\in B)$ be  a dispersed-edge or a lonely-edge in $G_i$, let $a_1$ be a neighbor of $b$ different from $a$, and let $e$ be an edge incident to $a_1$. If $e\in E(G_i)$, then $e$ is lonely.
\end{Lemma}
\begin{proof}
 Since $\mathscr{F}$ is a $B$-singular decomposition, $d_{G_i}(b)\leq1$, and hence,  $a_1b\notin E(G_i)$. Since $e\in E(G_i)$, we have that $e\neq a_1b$. It follows that $e$ cannot be a triplex-edge. If $e$ is a paired-edge or a dispersed-edge, then by switching $a_1b$ and $ab$ in the decomposition  $ \mathscr{F}$, we get a $B$-singular decomposition $ \mathscr{F}^\prime$ which has one more type-1 vertices than $\mathscr{F}$, contradicting our choice of $ \mathscr{F}$.  Therefore, $e$ is lonely.
\end{proof}

\begin{Lemma}
\label{le1}
\begin{enumerate}[(1)]
\item Every lonely-edge in $H_i$ belongs to $G_i$;
\item for every two adjacent edges in $H_i$, at least one of them is not a lonely-edge; and
\item  if $v_1v$ and $vv_2$ are two adjacent edges in $H_i$ such that neither is lonely, then $v \in B$.
\end{enumerate}
\end{Lemma}

\begin{proof}

Part (1) and part (2) follow immediately from Definition ~\ref{d1} and the definition of $H_i$. 

%Recall that $H_i$ is spanned by the endpoints of all lonely-edges in $G_i$. So every vertex in $H_i$ is incident to a lonely-edge of $G_i$.  By the definition of lonely-edges, a vertex in $A$ is incident to at most one lonely-edge in $G_i$. Suppose that $e=ab$ ($a \in A$) is a lonely edge in $H_i$ but not in $G_i$.  Since $a \in V(H_i)$, $a$ is also incident to a lonely-edge in $G_i$, a contradiction. This proves part (1).  Part (2) follows immediately from part (1).

To prove part (3), we suppose that $v\in A$. By the definition of $H_i$, there are three distinct lonely-edges $e_1, e_2, e\in E(H_i)$ incident to $v_1,v_2$ and $v$, respectively. By switching $e_1$ and $v_1v$, $e_2$ and $vv_2$ in $ \mathscr{F}$, we get another $B$-singular decomposition $\mathscr{F}^\prime$ which has one more type-1 vertices than $\mathscr{F}$, contradicting our choice of $\mathscr{F}$.
\end{proof}

 An {\em alternating cycle} of $H_i$ is a cycle in which for every pair of adjacent edges, exactly one of them is lonely. Similarly,  we define {\em alternating paths}.  A {\em rooted tree} is a pair $(T,r)$ where $T$ is a tree and $r \in V(T)$. The vertex $r$ is called the {\em root} of $(T,r)$. A rooted tree $(T,r)$ in $H_i$ is {\em alternating} if for each vertex $v\in V(T)$, the path from the root $r$ to $v$ is an alternating path.

\begin{Lemma}
\label{le3}
Let $P$ be a path from $b$  to $a$  in $H_i$ where $b\in B, a\in A$ and let $e$ be the edge of $P$ that is incident to $a$.
If $e$ is a lonely-edge, then $P$ is an alternating path.

\end{Lemma}
\begin{proof}
Assume that $P=b_0a_0b_1a_1\cdots b_na_n$ where $b_0=b, a_n=a, a_t\in A$ and $ b_t\in B$ for $0\leq t\leq n$.
Since $e=b_na_n$ is lonely, by Lemma~\ref{le1} part (2),  $a_{n-1}b_n$ is not lonely. Now by  Lemma~\ref{le1} part (3), the edge $b_{n-1}a_{n-1}$ is lonely. Repeating these arguments along the path $P$, we get that the path $P$ is alternating.
\end{proof}

We also require the following result of Bensmail, Lagoutte, and Valicov ~\cite{[BLV]} on the structure of $H_i$. We present a proof for completeness of the paper.

\begin{Lemma}
\label{le0}
Every connected component of $H_i$ has at most one  cycle. Moreover, if $C$ is a cycle in a connected component  of $H_i$, then $C$ must be alternating.
\end{Lemma}
\begin{proof}  First we show that all cycles in $H_i$ are alternating.
Let $C$ be a cycle in $H_i$. Assume that $C=b_0a_0b_1a_1\cdots b_na_nb_0$,  where $a_t\in A$,  $ b_t\in B$ for $0 \leq t \leq n$. By Lemma ~\ref{le1} part (3), exactly one of the two edges $b_na_n$ and $a_nb_0$ is a lonely-edge. Without loss of generality, assume that $b_na_n$ is a lonely-edge. By Lemma~\ref{le3}, the path $P=b_0a_0 b_1 a_1 \cdots b_n a_n$ is an alternating path. It follows that $C$ is an alternating cycle.

%Assume that none of edge on $C$ is lonely. By Lemma~\ref{le1}(2), $a_t\in B$ for $0\leq t\leq n$. It is a contradiction. So, there exists at least one lonely edge. We assume, without loss of generality, that $b_na_n$ is lonely. Let $p=b_0a_0b_1a_1\cdots b_{n-1}a_{n-1}b_na_n$. By Lemma~\ref{le3}, $P$ is alternating and $b_0a_0$ is lonely. By Lemma~\ref{le1}(1), $a_nb_0$ is non-lonely. Therefore, $C$ is alternating.

Next we show that every connected component of $H_i$ has at most one cycle. Suppose by contradiction that there are two cycles $C_1$ and $C_2$ in a connected component of $H_i$. We complete our proof in three cases. In each case, we get a contradiction. 

\medskip
{\em Case 1:}  $C_1$ and $C_2$ share only one vertex $v$. 

Since all cycles in $H_i$ are alternating, $v$ has two incident lonely-edges, contradicting Lemma~\ref{le1} part (2).

\medskip
{\em Case 2:} $C_1$ and $C_2$ share a path $P$ between $v_1$ and $v_2$. 

Since all cycles in $H_i$ are alternating and every vertex of $H_i$ is incident to exactly one lonely-edge,   the edge of $P$ incident to $v_1$ is a lonely-edge.  It follows that the other two edges in $C_1$ and $C_2$ incident to $v_1$ are non-lonely. By Lemma~\ref{le1} part (3), $v_1\in B$. Similarly,  the edge of $P$ incident to $v_2$ is a lonely-edge and $v_2\in B$. Since $v_1, v_2 \in B$, the path $P$ must be of even length; on the other hand, $P$ is an alternating path that starts and ends with a lonely-edge, so $P$ must be of odd length, a contradiction.

\medskip
{\em Case 3:}  $C_1$ and $C_2$ are joined by a path in $H_i$. 

 Let  $P$ be a path from $u_1$ to $u_2$ with  $u_1\in V(C_1), u_2\in V(C_2)$.  Choose this path $P$ to be a shortest one. So,  $d_{H_i} (u_1) \geq 3$. Since all cycles in $H_i$ are alternating, exactly one of the two edges of $C_1$  incident to $u_1$ is lonely.  Therefore $u_1$ is incident to at least two non-lonely-edges. By Lemma~\ref{le1} part (3), $u_1\in B$. Similarly, we conclude that $u_2\in B$.  Let $e= u_2w$ be the lonely-edge on $C_2$ incident to $u_2$  and let $P^\prime = P \cup \{e\}$. By Lemma ~\ref{le3}, $P^\prime$ is an alternating path. In particular, the edge of $P$ incident to $u_1$ is a lonely-edge, a contradiction.
\end{proof}

\begin{Lemma}
\label{le2}
Let $C$ be a cycle in a connected component  of $H_i$. Suppose that $ab\in E(C)$ ($a\in A$ and $ b\in B$) is a lonely-edge. Let $b^\prime$ be the neighbor vertex  of $a$ outside $C$ and  let $e$ be an edge incident to $b^\prime$. If $e\in E(G_i)$, then $e$ is a  triplex-edge.
\end{Lemma}

\begin{proof}
 Suppose otherwise that $e$ is not a triplex-edge.  Since $ab$ is a lonely-edge in $G_i$, we know that $ab^\prime\notin E(G_i)$ by the definition of a lonely-edge, and hence, $ab^\prime \neq  e$.  Assume that the vertices $b$, $b^\prime$, and $b_0$ are the three neighbors of $a$. Let $a_0b_0$ be the lonely-edge in $C$ that is incident to $b_0$. By switching $a_0b_0$ and $b_0a$, $ab^\prime$ and $e$ in the decomposition  $ \mathscr{F}$, we get a $B$-singular decomposition $ \mathscr{F}^\prime$ which has one more type-1 vertices than $\mathscr{F}$, contradicting our choice of  $\mathscr{F}$. This proves that $e$ is a triplex-edge.
\end{proof}

\begin{Lemma}
\label{le20}
Let $C_1$ and $C_2$ be two  cycles in $H_i$. If $C_1$ and $C_2$ can be joined by a path $P=a_1v_1vv_2a_2$ in $G$ where $a_t\in A \cap V(C_t)$ and $v_t \notin V(C_1) \cup V(C_2)$  for $t \in 
\{1, 2\}$, then $d_{G_i}(v)\neq 3$.
\end{Lemma}

\begin{proof}
Suppose otherwise that $d_{G_i}(v)=3$. Then by Lemma~\ref{le2}, $v_1v$ and $v_2v$  are triplex-edges. For $t \in \{1,2\}$, let $b_t$ and $b_t ^\prime$ be the two neighbors of $a_t$ on the cycle $C_t$;  assume without loss of generality that $a_t b_t$ is lonely and $a_t b_t ^ \prime$ is not lonely; let $a_t ^\prime b_t ^\prime $ be the lonely-edge on $C_t$ that is incident to $b_t ^\prime$. Now  by switching $a^\prime_1b^\prime_1$ and $a_1b^\prime_1$, $a_1v_1$ and $v_1v$, $a^\prime_2b^\prime_2$ and $a_2b^\prime_2$, $a_2v_2$ and $v_2v$ in the decomposition  $\mathscr{F}$, we get another $B$-singular decomposition $\mathscr{F}^\prime$; in $\mathscr{F} ^\prime$, the two vertices $a_1$ and $a_2$ are now type-1 vertices and the vertex $v$ is no longer a type-1 vertex. So $\mathscr{F} ^\prime$ has  one more type-1 vertices than $\mathscr{F}$, contradicting our choice of $\mathscr{F}$. This proves that $d_{G_i}(v)\neq 3$.
\end{proof}

\begin{Lemma}
\label{le21}
 Suppose that $C= b_0a_0\cdots b_{k-1}a_{k-1}b_0$ is a cycle with length $2k$  in $H_i$ where $a_t\in A$,  $b_t\in B$, $b_ta_t$ is lonely,  and $a_tb_{t+1}$ is not ($0\leq t\leq k-1$ and $t$ is taken modulo $k$). If for some $t$, $a_t$ and $a_{t+1}$  share a common neighbor vertex different from $b_{t+1}$, then $k$ is even.
\end{Lemma}

\begin{proof}
Suppose that $k$ is odd. Recall that $\mathscr{C}=\{C\in C_{2k}:$ $k$ is odd and $C$ is a cycle in $H_i$ for some $1\leq i\leq \Delta\}$ and $\mathscr{F}$ is chosen to maximize the number of type-1 vertices first, and then maximize the number of type-2 vertices, and finally minimize $|\mathscr{C}|$.

Let $b \neq b_{t+1}$ be a common neighbor of $a_t$ and $a_{t+1}$. Since $b_ta_t$ is lonely, for some $j\neq i$, the edges $a_tb_{t+1}$ and $a_tb$ are paired-edges in $G_j$. By switching $a_tb_{t+1}$ and $a_{t+1}b_{t+1}$, we get another $B$-singular  decomposition  $ \mathscr{F}^\prime$.  In $\mathscr{F}^\prime$,  both $a_t$ and $a_{t+1}$ are still type-2 vertices, so $\mathscr{F}^\prime$ have the same set of type-1 vertices and the same set of type-2 vertices as $\mathscr{F}$. However, in $\mathscr{F}^\prime$,  the edges $a_tb$ and $a_{t+1}b_{t+1}$ are lonely-edges in $G_j$, and the cycle $C^\prime$ induced by the endpoints of $a_tb$ and $a_{t+1}b_{t+1}$ is of length $4$.  By Lemma ~\ref{le0}, $C^\prime$ is the only new cycle in $H_j$ with respect to the decomposition $\mathscr{F} ^\prime$.  So in $\mathscr{F}^\prime$,  the cycle $C$ is removed from  $\mathscr{C}$ and no new cycle of length $2k^\prime$ ($k^\prime$ is odd) is added to $\mathscr{C}$.  This is contradicting our choice of $\mathscr{F}$.
\end{proof}

Now we look at a connected component of $H_i$ that has no cycle.

\begin{Lemma}
\label{le6}
Let $Q$ be a connected component of $H_i$.  Suppose that $Q$ has no cycle. Then there exists exactly one vertex $u \in A \cap V(Q)$ such that $d_Q(u)=1$.
\end{Lemma}

\begin{proof}
First we show that there exists at least one vertex in $A \cap V(Q)$ with degree 1 in $Q$.  By  Lemma~\ref{le1} part (3), no vertex in $A \cap V(Q)$ can have degree 3 in $Q$. So we assume that every vertex in $A \cap V(Q)$ has degree two in $Q$.  Let $a_0$ be an arbitrary vertex in $A \cap V(Q)$. Then in $Q$,  the vertex $a_0$ is incident to a lonely-edge and a non-lonely-edge. Let $a_0 b_0$ be the non-lonely-edge incident to $a_0$.  Since $b_0 \in V(Q)$, the vertex $b_0$ is incident to a lonely-edge, say $b_0 a_1$. Since $d_Q(a_1) \neq 1$,  by Lemma~\ref{le1} part (2), $a_1$ is also incident to a non-lonely-edge, say $a_1b_1$; by repeating this process, we get an alternating walk $a_0 b_0  a_1 b_1 \cdots$ in $Q$. Since $Q$ is a finite graph, there exists integers $l < m$ such that $a_l = a_m$. Therefore, $Q$ has a cycle, a contradiction.

Next we assume that $u_1, u_2 \in A \cap V(Q)$ are both vertices of degree 1 in $Q$. For $t \in \{1,2\}$, let $u_tv_t$ be the lonely-edge incident to $u_t$. Let $P$ be the path in $Q$ from $u_1$ to $u_2$. Since $u_1$ and $u_2$ both have degree 1 in $Q$, the vertices $v_1$ and $v_2$ are in $V(P)$.  Let $P^\prime=P-\{u_1v_1\}$. Since $u_2v_2$ is lonely, by Lemma~\ref{le3}, $P^\prime$ is alternating and the edge incident to $v_1$ in $P^\prime$ is lonely.  So $v_1$ is incident to two lonely-edges, contradicting Lemma~\ref{le1} part (2).
\end{proof}

Let $Q$ be a connected component of $H_i$. We now define a rooted tree $(T,r)$ as follows. On the one hand, if $Q$ has no cycle, then by Lemma~\ref{le6}, there exists a unique vertex $u \in A\cap V(Q)$ of degree 1 in $Q$. Let $r\in V(Q)$ be the unique neighbor of $u$ in $Q$. Clearly $r \in B$.  We define the rooted tree $(T,r) = (Q, r)$. 
On the other hand, if $Q$ has a cycle $C$. By Lemma~\ref{le0},  each connected component $Q^\prime$ of $Q-E(C)$ is a tree. Since $Q$ has exactly one cycle, $Q^\prime$ meets $C$ at exactly one vertex $r$. For each nontrivial connected component $Q ^\prime$, we define a rooted tree $(T,r)=(Q', r)$ where $r \in V(C)$.

\begin{Lemma}
\label{le5}
In the rooted tree  $(T,r)$ defined above, each of the following holds.
\begin{enumerate}[(1)]
\item The root $r$ is in  $B$.
\item $T$ is alternating.
\item All leaves of the tree $T$, except $u$, are in  $B$.
\item For each lonely-edge $ab\in $E(T) where $a\in A \backslash \{u\}$, we have that $b$ is the only child of $a$.
\end{enumerate}
\end{Lemma}

\begin{proof}
(1) By our definitions of $(T,r)$, we may assume that $Q$ has a cycle $C$ and $T=(Q ^\prime, r)$ with $r \in V(C)$.  
 Let $v_1$ be a child of $r$. Since $r\in V(C)$, the vertex $r$ is incident to a lonely-edge  and a non-lonely-edge in $C$. By the definition of lonely-edge, $rv_1$ is non-lonely. So $r$ is incident to two non-lonely-edges. By Lemma~\ref{le1} part (3), $r\in B$.

(2) Suppose otherwise that $T$ is non-alternating. Then there exists  a vertex $v\in V(T)$ such that the path $P$ from the root $r$ to $v$ is non-alternating. By Lemma~\ref{le1} part (2), there can not be two adjacent lonely-edges in $H_i$. Therefore, there must be two adjacent  edges $e_1$ and $e_2$  with the common vertex $w$ in the path $P$ such that neither $e_1$ nor  $e_2$ is lonely. By Lemma~\ref{le1} part (3), $w\in B$. Since $r\in B$, without loss of generality, let $P=b_0a_0b_1a_1\cdots b_{n-1}a_{n-1}b_na_n\cdots v$ where $b_0=r, a_{n-1}b_n=e_1, b_na_n=e_2$, and $b_n=w$. By Lemma~\ref{le1} part (3), the edge $b_{n-1}a_{n-1}$ is a lonely-edge. By Lemma~\ref{le3}, the sub-path of $P$ from $r$ to $a_{n-1}$ is an alternating path, in particular, the edge $r a_0$ is a lonely-edge. Now the vertex $r$ is incident to two lonely-edges in $H_i$, contradicting Lemma~\ref{le1} part (2). This proves that $T$ is alternating.

(3) Suppose otherwise that there exists a leaf $l\neq u $ of the tree $T$ such that  $l\in A$. Since $r\in B$, without loss of generality, we denote the path from the root $r$ to $l$ by $P=b_0a_0b_1a_1\cdots b_{n-1}a_{n-1}b_na_n$ where $b_0=r, a_n=l$. Since $a_n$ is a leaf,  $b_na_n$ is lonely. By Lemma~\ref{le3},  $P$ is alternating and the edge $b_0a_0 =ra_0$ is lonely. Now $r$ is incident to two lonely-edges in $H_i$, a contradiction. Therefore, all leaves of the tree $T$, except $u$, are in $B$.

(4) Since $a\in A$ and $a\neq u$, we know that $a$ is not a leaf of the tree $T$ by $(3)$. So, $a$ has at least one child. Assume that $a$ has two children.  Since $a\in A$ and $r\in B$, clearly $a\neq r$. So, $a$ has a parent. It follows that the three edges incident to $a$ are all in $H_i$. So in $H_i$, the vertex $a$ is incident to two non-lonely-edges, contradicting Lemma~\ref{le1} part (3). Therefore, $a$ has only one child.

Next we show that $b$ is the child of $a$. Suppose otherwise that $b$ is the parent of $a$.
Since $r\in B$, without loss of generality, we denote the path from the root  $r$ to $a$  by $P=b_0a_0\cdots b_na_n$ where $b_0=r, b_n=b, a_n=a$. Since $b_na_n=ba$ is lonely,  $P$ is alternating and the edge $b_0a_0$ is lonely by Lemma~\ref{le3}.  Once again, $r$ is incident to two lonely-edges in $H_i$, a contradiction. It follows that $b$ is the only child of $a$.
\end{proof}

\section{proof of theorem ~\ref{thm:main}} \label{sec:proof}

In this section we prove Theorem ~\ref{thm:main}. Let $G ^\prime$ be a subgraph $G$.  We say that $G ^\prime$ has a {\em strong partial edge-coloring} for $G$ if we can assign colors to all edges of $G^\prime$ such that every pair of edges that can see each other in $G$ receive different colors.  To prove Theorem~\ref{thm:main}, it is sufficient to show that, for each $i \in \{1, 2, \cdots, \Delta\}$, the graph $G_i$ has a strong partial edge-coloring for $G$ using at most three colors. 

We first solve the case when $H_i$ has no cycle. Here we would like to point out that the coloring scheme used in the next lemma is precisely the same as the scheme used in ~\cite{[BLV]}; the only difference is that they use an arbitrary vertex as the root (see Phrase 2, Page 7 in ~\cite{[BLV]}) while we choose a very special vertex to be the root.  This way we avoid using the extra color that was required in ~\cite{[BLV]}.  

\begin{Lemma} \label{lemma:no-cycle}
If the graph $H_i$ has no cycle, then the graph $G_i$ has a strong partial edge-coloring  for $G$ using at most three colors.
\end{Lemma}
\begin{proof}  Let $Q$ be a connected component of $H_i$. Then $Q$ is a tree. By Lemma~\ref{le6}, there is a unique vertex $u \in A \cap V(Q)$ such that $d_Q(u)=1$. Let $r\in V(Q)$ be the unique neighbor of $u$ in $Q$ and let the rooted tree $T_Q= (Q, r)$. We use the greedy coloring algorithm to color the edges in $G_i$ in the following order: first color all triplex-edges, then all paired-edges, then all dispersed-edges; finally for all lonely-edges, color them in the order as they are encountered during a Breadth-First Search (BFS) algorithm performed on $T_Q =(Q,r)$ for every component $Q$ of $H_i$.  We show that this coloring procedure requires only three colors. Let $e$ be an arbitrary edge of $G_i$.

If $e$ is a triplex-edge, then clearly in $G_i$, the edge $e$ sees at most two other edges that are already colored; i.e., the two edges adjacent to $e$, so all triplex-edges can be colored using three colors.

Now let $e=ab$ $(a\in A, b\in B)$ be a paired-edge and let $b^\prime$ and $b^{\prime\prime}$ be the two other neighbors of $a$. Assume $ab^\prime$ is a paired-edge and $ab^{\prime\prime}$ is a lonely-edge. Let $a^\prime \neq a$ be a neighbor of $b$. Since $ab \in E(G_i)$, we have that $a^\prime b \notin E(G_i)$. So, $a ^\prime$ cannot be incident to triplex-edges in $G_i$. Furthermore, $a ^\prime$ cannot be incident to two pair-edges in $G_i$, as otherwise, we may switch $ab$ and $a^\prime b$ and the resulting decomposition would have one more type-1 vertex $a ^\prime$, contradicting our choice of $ \mathscr{F}$. Therefore, in $G_i$,  the edge $e$ sees $ab^\prime$ and possibly another triplex-edge incident to $b^{\prime\prime}$, so all paired-edges can be colored using at most three colors.

Next let $e=ab$ where $a\in A, b\in B$ be a dispersed-edge and let  $b^\prime$ and $b^{\prime\prime}$ be the two other neighbors of $a$.  For every neighbor $a^\prime \neq a$ of $b$,  if $a^\prime$ is incident to an edge $e ^\prime$ in $G_i$, then by Lemma ~\ref{le00}, $e^\prime$ is a lonely-edge, and hence, it is not yet colored.  Therefore, in $G_i$, the edges that $e$ can see and that were already colored must be incident to $b^\prime$ or $b^{\prime\prime}$; and there are at most two such edges. So all dispersed-edges can be colored using at most three colors.

Finally let $e=ab$ $(a\in A, b\in B)$  be a lonely-edge in a tree $T_Q$ where $Q$ is a connected component of $H_i$.
Let $a^\prime \neq a$ be  a neighbor of $b$. Let $e^\prime \in E(G_i)$ be an edge incident to $a ^\prime$.  By Lemma~\ref{le00},  $e^\prime$ is lonely.  If $b$ is not the root of $T_Q$, then by Lemma ~\ref{le5}, $b$ is the child of $a$. By the construction of the rooted tree $T_Q$,  we know that $a^\prime$ is a child of $b$. So, $e^\prime$ is not yet colored.
 If $b$ is the root of $T_Q$, then $a$ is a leaf of $T_Q$ and $a^\prime$ is a child of $b$. So we also get that $e^\prime$ is not yet colored.
 It follows that an edge of $G_i$ that $e$ can see must be incident to one of the two other neighbors of $a$, and hence,  there are at most two such edges.  Therefore, all lonely-edges can be colored using at most three colors.
\end{proof}

The case when $H_i$ has a cycle is more involved. In Bensmail {\em et al.} \cite{[BLV]}, all triplex-edges are colored in an arbitrary order. We will require a special ordering on all triplex-edges that are of distance 1 from a cycle in $H_i$ and a special ordering on all lonely-edges.  Let $a$ be a type-1 vertex. Then by Lemma~\ref{le20}, at most one of the three triplex-edges incident to $a$ can be at distance 1 from a cycle in $H_i$. Therefore, we may assume that $H_i$ is connected. Let $C$ be the unique cycle in $H_i$.   Assume that $C=b_0a_0b_1a_1\cdots b_{k-1}a_{k-1}b_0$,  where $a_t\in A$, $b_t\in B$, the edge $b_ta_t$ is lonely,  and the edge $a_tb_{t+1}$ is non-lonely ($t$ is taken modulo $k$). For each $a_t$, let $b_t ^\prime$ be the neighbor of $a_t$ that is not on the cycle $C$.  If there exists $0 \leq t < k$ such that $b_t ^\prime$ is not incident to a triplex-edge in $G_i$, then we assume without loss of generality that $b_0^\prime$ is not incident to a triplex-edge in $G_i$. For each $0 \leq t <k$, let $T_t =(T,r)$ be the rooted tree where  $T$ is a connected component of $H_i-E(C)$ and $r=b_t \in V(T) \cap V(C)$.

Now we describe a coloring procedure for $G_i$.
 If $b_0 ^\prime$ is not incident to a triplex-edge in $G_i$,
 or $b_0^\prime= b_1^\prime=\cdots= b_{k-1} ^\prime$, then we color all triplex-edges greedily in an arbitrary order.

Next assume that each $b_t ^ \prime$ is incident to a triplex-edge in $G_i$ and there exists $j$ such that $b_j  ^\prime \neq b_{j+1} ^\prime$. Assume without loss of generality that $b_0 ^\prime \neq b_1 ^\prime$. For $k \in \{0, 1\}$, let $e_k^\prime$ be the triplex-edge in $G_i$ that is incident to $b_k ^\prime$.  Note that since the graph is bipartite, $e_0 ^\prime \neq e_1^\prime$ (but $e_0^\prime$ may be adjacent to $e_1^\prime$).

In the case where $e_0^\prime$ and $e_1^\prime$ are adjacent, let $e_0 ^{\prime\prime}$ be the other triplex-edge adjacent to $e_0^\prime$ and $e_1^\prime$. We will first color the edges $e_0^\prime$, $e_1^\prime$, and $e_0 ^{\prime\prime}$ greedily in the given order, then color all other triplex-edges.

 In the case where $e_0^\prime$ and $e_1^\prime$ are not adjacent, we let $e_t^{\prime\prime}$ and $e_t ^{\prime\prime\prime}$ be the two triplex-edges adjacent to $e_t^\prime$ where $t \in \{0, 1\}$. We will first color the edges $e_0^\prime$, $e_0 ^{\prime\prime}$, $e_0 ^{\prime\prime\prime}$, $e_1 ^{\prime\prime}$, $e_1^\prime$, $e_1 ^{\prime\prime\prime}$ greedily in the given order, then color all other triplex-edges.

Once all triplex-edges are colored, we color all paired-edges, then all dispersed-edges. Finally for all lonely-edges, we color them in the following order:  color the lonely-edges on the cycle $C$ in the order $b_1a_1$, $b_2a_2$, $\cdots$, $b_{k-1} a_{k-1}$, $b_0a_0$;  note that, for $t \neq t^\prime$, the lonely-edges in $T_t$ do not see the lonely-edges in $T_{t^{\prime}}$,  so we may color the rooted trees $T_t$'s in an arbitrary order, but within each rooted tree $T_t$, 
the lonely-edges in $T_t$ are colored in the order determined by the Breadth-First Search (BFS) algorithm starting at the root $b_t$.

\begin{Lemma} \label{lemma:has-cycle}

For each $G_i$, the coloring procedure described above produces a strong partial edge-coloring for $G$ using at most three colors.
\end{Lemma}

\begin{proof} Let $e$ be an arbitrary edge of $G_i$. If $e$ is a triplex-edge, a paired-edge, or a dispersed-edge, then  the proof for Lemma~\ref{lemma:no-cycle} shows that there exists an available color for $e$. So we assume that $e$ is a lonely-edge.   We first look at lonely-edges on the cycle $C$.

For each vertex $b_t$ ($0\leq t\leq k-1$) on $C$, we let $a^\prime_t$ be a neighbor of $b_t$ outside $C$. If there exists an edge $e_t$ of $G_i$  incident to $a^\prime_t$, then by Lemma~\ref{le00},  $e_t $ is lonely.  So $e_t$ is in $T_{t}$, and hence, it is not yet colored.

Assume that $t\neq 0$. Then $b_{t+1}a_{t+1}$ ($t+1$ is taken modulo $k$) is not yet colored. Therefore, $b_ta_t$ sees at most two colored edges in $G_i$:  the edge $b_{t-1} a_{t-1}$ and possibly a triplex-edge incident to $b_t ^\prime$. So there exists a color for $b_ta_t$.

Next assume that $t=0$. If $b_0 ^\prime$ is not incident to a triplex-edge in $G_i$, then $b_0a_0$ only sees two colored edges:  $b_{k-1}a_{k-1}$ and $b_1a_1$. So,  there exists at least one available color for $e$. So we may assume that each $b_t ^\prime$ is incident to a triplex-edge $e_t$.

 If $b_0^\prime= b_1^\prime=\cdots= b_{k-1} ^\prime$, then by Lemma~\ref{le21}, $k$ is even. Note that, all lonely-edges on $C$ can see the triplex-edge $e_0$. So the lonely-edges on $C$ are colored alternatively using two colors different from the color assigned to $e_0$. Since $k$ is even, the edges $b_1a_1$ and $b_{k-1}a_{k-1}$ are assigned the same color, so there exists a color available for $e=b_0a_0$.

Next we assume that $b_0^\prime \neq b_1 ^\prime$. Then our coloring procedure assigns to $e_0^\prime$ and $e_1^\prime$ different colors. When we color $b_1a_1$,  since $b_1a_1$ only sees one colored edge $e_1^\prime$,  it is assigned the same color as $e_0^\prime$  by the greedy coloring.  Note that, $b_0a_0$ sees three colored edges: $b_{k-1}a_{k-1}$, $b_1a_1$, and $e_0^\prime$. Since $b_1a_1$ and $e_0^\prime$ are assigned the same color, there exists a color available for $b_0a_0$.

Now we have assigned colors to all lonely-edges on the cycle $C$ and we are left to assign colors to lonely-edges in each rooted tree $T_t$.
 Let $e=ab$ $(a\in A, b\in B)$  be a lonely-edge in a tree $T_t$.
Let $a^\prime \neq a$ be  a neighbor  of $b$. Let $e^\prime \in E(G_i)$ be an edge incident to $a ^\prime$.  By Lemma~\ref{le00},  $e^\prime$ is lonely. By Lemma ~\ref{le5}, $b$ is the child of $a$. By the construction of the rooted tree $T_Q$, we know that $a^\prime$ is a child of $b$. So, $e^\prime$ is not yet colored.
 It follows that the edges in $G_i$ that the edge $e$ can see must be incident to one of the two other neighbors of $a$, and there are at most two such edges.  Therefore, there exists a color available for $e$.
\end{proof}

\section{Possible extensions}  \label{sec:ext}

So now we know that the strong chromatic index of a $(3,\Delta)$-bipartite graph is at most $3 \Delta$. A natural question to ask is whether the proof technique can be used to prove similar results for other classes of bipartite graphs. More specifically, what about the class of $(4, \Delta)$-bipartite graphs? Greedy coloring arguments can show that the strong chromatic index of a $(4, \Delta)$-bipartite graph is at most $7 \Delta  -3$.  Brualdi and Quinn Massey \cite{[BQ]} conjectured that the bound should be $4 \Delta$. We tried without success to apply the same technique to $(4,\Delta)$-bipartite graphs.  

Let $G$ be a $(4, \Delta)$-bipartite graph with bipartition $(A,B)$ and $\Delta(A) = 4$. We decompose $G$ in a similar manner, say $\{ G_1, G_2, \cdots, G_{\Delta} \}$.  A vertex $a$ in $A$ may be one of the following five types. 

\begin{itemize}
\item Type-1: all four edges incident to $a$ are in the same $G_i$;

\item Type-2: there exist distinct  $i$ and $j$ such that three edges incident to $a$ are in $G_i$, and the other one is in $G_j$; 

\item Type-3: there exist distinct  $i$ and $j$ such that two edges incident to $a$ are in $G_i$, and the other two are in $G_j$; 

\item Type-4: there exist distinct $i$, $j$, and $k$ such that two edges incident to $a$ are in $G_i$, one is in $G_j$, and one is in $G_k$; 

\item Type-5: each edge incident to $a$ is in a different $G_i$. 
\end{itemize}  

We chose a decomposition to maximize the number of  vertices of Type-1, Type-2, Type-3, and Type-4, in the given order.  
We realized that the hardest case was how to color all lonely-edges incident to Type-2 vertices.  Recall that in a $(3,\Delta)$-bipartite graph, every vertex in $A \cap V(H_i)$ has degree 1 or 2 in $H_i$. This is no longer the case in a $(4,\Delta)$-bipartite graph. Indeed, let $H_i$ be the graph spanned by the endpoints of all lonely-edges in $G_i$ that are incident to Type-2 vertices; then a vertex in $A \cap V(H_i)$ may have degree $1$, $2$, or $3$ in $H_i$. Thus, the structure of $H_i$ is not very clear to us. We strongly feel that the current approach is unlikely to be applicable. New ideas and techniques may be necessary even for the next case of $(4,\Delta)$-bipartite graphs. 

Another possible extension of our result is related to list edge-coloring. Let $G$ be a graph and let $L(e)$ be the list of available colors for each edge $e \in E(G)$. A {\em strong list edge-coloring} of $G$ is a strong edge-coloring of $G$ such that every edge $e$ receives a color from its list. A graph $G$ is called {\em strongly $k$-edge-choosable} if it has a strong list edge-coloring whenever each edge has a list of order at least $k$.  

We do not have evidence against that a $(d_A, d_B)$-bipartite graph is strongly $d_A d_B$-choosable, while whether it is strongly $d_A d_B$-colorable is still wide open.   But even for $(d_A, d_B)$-bipartite graphs with $(d_A, d_B)\in \{(3,3), (2,\Delta), (3,\Delta)\}$,  which are known to be strongly $d_A d_B$-edge-colorable,  it is not clear to us if they are strongly $d_A d_B$-edge-choosable, since the current edge decomposition method does not seem to work any more.  We propose to start a new research line by studying the following three open problems. Again we may need to develop some new techniques to solve these problems. 

\begin{itemize} 
\item Problem 1: Are $(3,3)$-bipartite graphs strongly $9$-edge-choosable?  

\item Problem 2: Are $(2, \Delta)$-bipartite graphs strongly $2\Delta$-edge-choosable? 

\item Problem 3:  Are $(3,\Delta)$-bipartite graphs strongly $3\Delta$-edge-choosable?  
\end{itemize}

\end{document}